\documentclass[11pt,a4paper]{article}
\usepackage[affil-it]{authblk}
\usepackage{graphicx}
\usepackage{moreverb,url}

\newcommand\BibTeX{{\rmfamily B\kern-.05em \textsc{i\kern-.025em b}\kern-.08em
T\kern-.1667em\lower.7ex\hbox{E}\kern-.125emX}}

\usepackage{a4wide}
\usepackage{algorithmic}
\usepackage{amsfonts}
\usepackage{amsmath}
\usepackage{amssymb}  
\usepackage{amsthm}
\usepackage{bold-extra}
\usepackage{braket}
\usepackage{cleveref}
\usepackage{epsfig} 
\usepackage{epstopdf}
\usepackage{graphics} 
\usepackage{graphicx}
\usepackage{mathptmx} 
\usepackage{stmaryrd}
\usepackage[caption=false]{subfig}
\usepackage{pgfplots}
\usepackage[most]{tcolorbox}
\usepackage{times} 

\colorlet{texcscolor}{blue!50!black}
\colorlet{texemcolor}{red!70!black}
\colorlet{texpreamble}{red!70!black}
\colorlet{codebackground}{black!25!white!25}

\newcommand{\Real}{\mathbb R}
\newcommand{\sgn}{\mathrm{sign}}

\newtheorem{thrm}{Theorem}[section]

\newtheorem{prpstn}[thrm]{Proposition}

\newtheorem{dfntn}[thrm]{Definition}

\newtheorem{rmrk}[thrm]{Remark}

\begin{document}

\title{A solution of the minimum-time velocity planning problem\\
  based on lattice theory%
}

\providecommand{\keywords}[1]{\textbf{\textit{Index terms---}} #1}
\author{Luca Consolini$^1$, Mattia Laurini$^1$, Marco Locatelli$^1$, Andrea Minari$^1$}

\date{\small $^1$ Dipartimento di Ingegneria e Architettura, Universit\`a degli Studi di Parma,\\ Parco Area delle Scienze 181/A, 43124 Parma, Italy.\\ luca.consolini@unipr.it, mattia.laurini@unipr.it, marco.locatelli@unipr.it, andrea.minari2@studenti.unipr.it}

\maketitle

\begin{abstract}
For a vehicle on an assigned path, we find the minimum-time speed
law that satisfies kinematic and dynamic constraints, related to
maximum speed and maximum tangential and transversal acceleration.
We present a necessary and sufficient condition for the feasibility of the
problem and a simple operator, based on the solution of two ordinary
differential equations, which computes the optimal solution.
Theoretically, we show that the problem feasible set, if not empty, is a
lattice, whose supremum element corresponds to the optimal solution.
\end{abstract}



\maketitle

\section{Introduction}

An important problem in motion planning is the computation of the
minimum-time motion of a car-like vehicle from a start configuration
to a target one while avoiding collisions (obstacle avoidance) and
satisfying kinematic, dynamic and mechanical constraints (for
instance, on velocities, accelerations and maximal steering
angle). This problem can be approached in two ways:
\begin{itemize}
\item[1.] as a minimum-time trajectory planning where both the
path to be followed by the vehicle and the timing law on this path
(i.e., the vehicle's velocity) are simultaneously designed; or
\item[2.] as a (geometric) path planning followed by a minimum-time
velocity planning on the planned path
(see for instance~\cite{doi:10.1177/027836498600500304}).
\end{itemize}

In this paper, following the second paradigm, we assume that
a path that joins the initial and final configurations, compatible
with the maximum curvature allowed for the car-like vehicle,
is assigned and we aim at finding the time-optimal speed law
that satisfies some kinematic and dynamic constraints.

Previous works address the problem in time domain.
For instance,~\cite{Velenis2008} investigates necessary optimality
conditions that allows deriving a semi-analytical solution.
Reference~\cite{7810621} assumes that the vehicle moves along a
specified clothoid and proposes a semi-analytical solution for the
optimal profile of the longitudinal acceleration.
Some approaches propose special speed profiles that guarantee the
satisfaction of kinematic and dynamic constraints (see for instance
\cite{MunOllPraSim:94, SolNun2006, Villagra-et-al2012, CheHeBuHanZha2014}).
Other works (see, for instance, \cite{doi:10.1177/027836498500400301,
Verscheure09, 88024, LiSunKurZhu2014, Nagy2017})
represent the speed law as a function $v$ of the arc-length
position $s$ and not as a function of time.

Here, we follow this second approach, that considerably simplifies the optimization problem.
This paper is a development of our previous works
(\cite{Minari16, MinSCL17}), that consider the following
problem
\begin{subequations}
\label{eqn_problem_pr}
\begin{align}
\min_{v \in W^{1,\infty}\left(\left[0,s_f\right]\right)} & \int\limits_0^{s_f} v^{-1}(s) d s &\label{obj_fun_pr}\\
& v^-(s)\leq v(s) \leq  v^+(s),& s \in [0,s_f], \label{con_speed_pr}\\
& \alpha^-(s) \leq 2 v'(s)v(s) \leq \alpha^+(s), &s \in [0,s_f], \label{con_at_pr}\\
& |k(s)| v(s)^2 \leq \beta(s),&s \in [0,s_f],  \label{con_an_pr}
\end{align}
\end{subequations}

Here, $v^-$, $v^+$, $\alpha^-$, $\alpha^+$ are assigned functions,
with $v^-$, $v^+$ non-negative.
The objective function~\eqref{obj_fun_pr} is the total maneuver
time and constraints~\eqref{con_speed_pr},~\eqref{con_at_pr},~\eqref{con_an_pr}
limit velocity and the tangential and normal components of acceleration.

Problem~\eqref{eqn_problem_pr} belongs to a class of problems that
includes optimal velocity planning for manipulators and has the form

\begin{equation}
\label{eqn_num_int}
\begin{aligned}
\min_{v \in W^{1,\infty}\left(\left[0,s_f\right]\right)} & \int\limits_0^{s_f} v^{-1}(s) d s \\
&a_i(s) \dot v(s) v(s) + b_i(s) v(s)^2 + c_i (s) \leq 0, \, i=1, \ldots, m\,,
\end{aligned}
\end{equation}
where $m$ is the number of constraints and $a_i,b_i,c_i$ are assigned functions.
It is clear that Problem~\eqref{eqn_problem_pr} is a special case of
Problem~\eqref{eqn_num_int}, obtained for a specific choice of
$a_i$, $b_i$, $c_i$.

Our previous works~\cite{Minari16, MinSCL17} present an algorithm,
with linear-time computational complexity with respect to the number of
variables, that provides an optimal solution of~\eqref{eqn_problem_pr}
after spatial discretization.
Namely, the path is divided into $n$ intervals of equal length and
Problem~\eqref{eqn_problem_pr} is approximated
with a finite dimensional one in which the derivative of $v$ is
substituted with a finite difference approximation.

In this paper, we compute directly the exact continuous-time solution of
Problem~\eqref{eqn_problem_pr} without performing a finite-dimensional
reduction.
The main result of the paper is presented in
Theorem~\ref{main-theorem}. It gives a sufficient and necessary
condition for the feasibility of Problem~\eqref{eqn_problem_pr} and
presents its optimal solution, which is computed as the pointwise minimum
of the solutions of two ODEs.

The method we propose presents some resemblances with the method of
``numerical integration'', introduced for problems of class~\eqref{eqn_num_int}.
For instance,~\cite{88024} proposes a
method, based on the identification of ``characteristic switching
points'' in which the maximum velocity is attained. This simplifies the
calculation of the optimal velocity profile. A related algorithm is
presented in~\cite{kunz2012time}. Recent
paper~\cite{7827149} presents various properties of numerical
integration methods.
Anyway, the method we propose is simpler and more efficient since it
leverages the special structure of Problem~\eqref{eqn_problem_pr} with
respect to the more general problem~\eqref{eqn_num_int}.

{\bf Statement of contribution:}
With respect to existing literature, the new contributions of this work are the following ones:
\begin{itemize}
\item It presents a necessary and sufficient condition for the
  feasibility of Problem~\eqref{eqn_problem_pr} (see part \textit{i.} of
  Theorem~\ref{main-theorem}).
\item It proposes a simple operator, based on the solution of two
  ordinary differential equations, that computes the optimal solution
  (see part \textit{ii.} of Theorem~\ref{main-theorem}).
\end{itemize}

Note that these results correspond to the generalization to the continuous-time
case of the results presented in~\cite{MinSCL17} for the
spatially-discretized version of Problem~\eqref{eqn_problem_pr}. In
fact, this paper shares some of its fundamental ideas
with~\cite{MinSCL17}. Namely:
\begin{itemize}
\item The feasible set of
Problem~\eqref{eqn_problem_pr} has the algebraic structure of a
lattice, if equipped with the operations of pointwise minimum and
maximum.
\item The optimal solution of Problem~\eqref{eqn_problem_pr}
corresponds to the supremal element of this lattice.
\item The optimal solution of Problem~\eqref{eqn_problem_pr} is
obtained with a projection operation and its optimality is proven by the
Knaster-Tarski Fixpoint Theorem.
\end{itemize}

Anyway, solving the problem in a function space requires various nontrivial
technical extensions to the proofs presented in~\cite{MinSCL17}.

\emph{Notation}: Given an interval $I = \left[ a, b \right]$
and a measurable function $f : I \rightarrow \Real$, let us recall that
\[
{\left\| f \right\|}_\infty =
\inf\left\{ C \geq 0 : |f(x)| \leq C \textrm{ for almost every } x \in I \right\},
\]
\[
L^\infty(I) = \left\{ f : I \rightarrow \Real \ : \  {\left\| f \right\|}_\infty < \infty \right\},
\]
and
\[
W^{1, \infty}(I) = \left\{ f \in L^\infty(I) : Df \in L^\infty(I)\right\},
\]
where $Df$ denotes the weak derivative of $f$.
Let $f, g: \ I \rightarrow \Real $, define $f \wedge g $, $f \vee g $,
as, respectively, the pointwise minimum and maximum operations;
moreover let us define the partial order $\leq $ as follows
\[
f \leq g \iff \forall x \in I \ f(x) \leq g(x).
\]
We will write ``almost everywhere'' as ``a. e.'' and we will use symbol
$\lightning$ at the end of a proof to state that a contradiction has been reached.

\emph{Paper organization}:
Section~\ref{sec_formulation} presents the addressed optimal control
problem. Section~\ref{sec_main} presents the main result
(Theorem~\ref{main-theorem}) and Section~\ref{sec_ex} presents
some examples. Finally, Section~\ref{sec_proofs} presents the
proof of the main result.

\section{Problem formulation}
\label{sec_formulation}

Let $\gamma: \left[ 0, s_f \right] \rightarrow \Real^2$ be a $C^2$ function such that
$\left\| \gamma^\prime(\lambda) \right\| = 1, \forall \lambda \in \left[ 0, s_f \right]$.
The image set $\gamma\left(\left[0,s_f\right]\right)$ represents the path followed
by a vehicle, $\gamma(0)$ the initial configuration and $\gamma(s_f)$ the final one.
We want to compute the speed-law that minimizes the overall transfer time while
satisfying some kinematic and dynamic requirements. To this end, let
$\lambda: \left[ 0, t_f \right] \rightarrow \left[ 0, s_f \right]$ be a differentiable
monotone increasing function that represents the vehicle position as a function
of time and let $v: \left[ 0, s_f \right] \rightarrow \left[ 0, +\infty \right]$
be such that, $\forall t \in \left[0,t_f\right],\dot{\lambda}(t) = v(\lambda(t)).$
In this way, $v(s)$ is the vehicle velocity at position s. The position of the vehicle
as a function of time is given by $x: \left[ 0, t_f \right] \rightarrow \Real^2$, $x(t)
= \gamma(\lambda(t))$, the velocity and acceleration are given by
\begin{gather*}
\dot{x}(t) = \gamma^\prime(\lambda(t))v(\lambda(t)),\\
\ddot{x}(t) = a_L(t)\gamma^\prime(\lambda(t)) + a_N(t)\gamma^{\prime\perp}(\lambda(t)),
\end{gather*}
where $a_L(t) = v^\prime(\lambda(t))v(\lambda(t))$ and $a_N(t)(t)= k(\lambda(t))v(\lambda(t))^2$
are, respectively, the longitudinal and normal components of acceleration.
Here $k: \left[ 0, s_f \right] \rightarrow \Real$ is the scalar curvature, defined as
$k(s) = \left\langle \gamma^{\prime\prime}(s),\gamma^{\prime}(s)^\perp \right\rangle.$

We require to travel the distance $s_f$ in minimum-time while
satisfying constraints on the vehicle velocity and on its longitudinal and normal
acceleration.
\begin{figure}
\centering
\includegraphics[width = 2.5in]{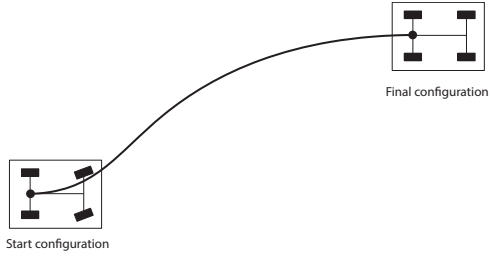}
\caption{A path to follow for an autonomous car-like vehicle.}
\label{fig:path-to-follow}
\end{figure}
The minimum-time problem can be
approached by searching a velocity profile
$v$ which is the solution of Problem~\eqref{eqn_problem_pr}.

It is convenient to make the change of variables $w=v^2$ (see also \cite{Verscheure09}), so that
Problem~\eqref{eqn_problem_pr} takes on the form
\begin{subequations}
\label{eqn_problem_pr_w}
\begin{align}
\min_{w \in W^{1, \infty}\left(\left[0,s_f\right]\right)} & \int\limits_0^{s_f} w(s)^{-\frac{1}{2}} d s &\label{obj_fun_pr_w}\\
& \mu^-(s) < w(s) \leq  \mu^+(s),& s \in [0,s_f], \label{con_speed_pr_w}\\
& \alpha^-(s) \leq w'(s) \leq \alpha^+(s), &s \in [0,s_f],  \label{con_at_pr_w}
\end{align}
\end{subequations}
where
\begin{equation}
\label{eqn_u}
\mu^+(s) = v^+(s)^2 \wedge \frac{\beta(s)}{k(s)},\qquad
\mu^-(s)= v^-(s)^2
\end{equation}
represent the upper bound of $w$,
(depending on the velocity bound $v^+$ and the curvature $k$) and
the lower bound of $w$, respectively.

In this paper, we actually address the following problem,
which is slightly more general than~\eqref{eqn_problem_pr_w},
\begin{subequations}
\label{eqn_problem_pr_wg}
\begin{align}
\min_{w \in W^{1,\infty}\left(\left[0,s_f\right]\right)} & \Psi(w) &\label{obj_fun_pr_wg}\\
& \mu^-(s) < w(s) \leq  \mu^+(s),& s \in [0,s_f], \label{con_speed_pr_wg}\\
& \alpha^-(s) \leq w'(s) \leq \alpha^+(s), &s \in [0,s_f],  \label{con_at_pr_wg}
\end{align}
\end{subequations}

where $\Psi: W^{1,\infty}\left(\left[0,s_f\right]\right) \to \Real$ is order reversing
(i.e., $x \geq y \Rightarrow \Psi(x) \leq \Psi(y)$) and
$\mu^-$, $\mu^+$, $\alpha^-$, $\alpha^+ \in L^\infty\left(\left[0,s_f\right]\right)$ are
assigned functions with $\mu^-, \alpha^+\geq 0$, $\alpha^-\leq 0$. Note that the objective
function~\eqref{obj_fun_pr_w} is order reversing, so that
Problem~\eqref{eqn_problem_pr_w} has the form~\eqref{eqn_problem_pr_wg}.
Consider the following:
\begin{dfntn}
\label{def:defQ}
Let $\mathcal{Q}$ be the subset of $W^{1,\infty}([0,s_f,])$ such
that $\mu \in \mathcal{Q}$ if $\sgn\left(\mu^\prime - \alpha^+\right)$
and $\sgn\left(\mu^\prime - \alpha^-\right)$ are Riemann integrable
(i.e., in view of the boundedness of the $\sgn$ function, a. e. continuous),
where $\sgn: \Real \rightarrow \{-1, 0, 1\}$ is defined as
\[
\sgn(x) =
\begin{cases}
1, & \mbox{if } \ x > 0 \\
0, & \mbox{if } \ x = 0 \\
-1, & \mbox{if } \ x < 0.
\end{cases}
\]
\end{dfntn}

\section{Main Results}
\label{sec_main}

Define the \emph{forward operator}
$F: \mathcal{Q} \rightarrow W^{1,\infty}\left(\left[ 0,s_f \right]\right)$
such that $F(\mu)=\phi$, where $\phi$ is the solution of the following
differential equation

\begin{equation}\label{eq:forward_iteration}
\begin{cases}
\phi'(x) =f(x,\phi)=
\begin{cases}
\alpha^+(x) \wedge \mu^\prime (x), & \mbox{if } \phi(x) \geq \mu(x) \\
\alpha^+(x), & \mbox{if }  \phi(x) < \mu(x)
\end{cases}\\
\phi(0) = \mu(0)\,.
\end{cases}
\end{equation}

Note that the solution of~\eqref{eq:forward_iteration} exists and is
unique by Theorem~1 in Chapter~2, Section~10
of~\cite{arscott2013differential}, since function $f$ is bounded on $[0,s_f]$, the subset of $[0,s_f] \times
\Real$ in which $f$ is discontinuous has zero measure and, $\forall x
\in [0,s_f]$, $\forall u,y \in \Real$,
\[
(u-y) (f(x,u)-f(x,y)) \leq 0.
\]

Conversely, define the \emph{backward operator} $B:
\mathcal{Q} \rightarrow W^{1,\infty}\left(\left[ 0,s_f \right]\right)$,
such that $B(\mu)=\phi$, where $\phi$ is the solution of

\begin{equation}\label{eq:backward_iteration}
\begin{cases}
\phi'(x) =
\begin{cases}
\alpha^-(x) \vee \mu^\prime (x), & \mbox{if } \phi(x) \geq \mu(x) \\
\alpha^-(x), & \mbox{if } \phi(x) < \mu(x)
\end{cases}\\
\phi(s_f) = \mu(s_f)\,,
\end{cases}
\end{equation}
whose existence and uniqueness hold for the same reasons as~\eqref{eq:forward_iteration}.

Finally, define the \emph{meet operator}  $M: \mathcal{Q}
 \rightarrow W^{1,\infty}\left(\left[ 0,s_f \right]\right)$ as
\begin{equation}
\label{eq:minimum}
M(\mu)=F(\mu) \wedge B(\mu).
\end{equation}

We claim that the \emph{meet operator} $M$ allows checking the
feasibility of Problem~\eqref{eqn_problem_pr_wg} and that, in case
Problem~\eqref{eqn_problem_pr_wg} is feasible, function
$v^* = M(\mu^+)$ represents its optimal solution.

Namely, the following is the main result of this paper.

\begin{thrm}\label{main-theorem}
Let $\mu^+ \in \mathcal{Q}$, then the following statements hold:
\begin{enumerate}
\item[\textit{i.}] Problem~\eqref{eqn_problem_pr_wg} is feasible if and only if function
$w^* = M(\mu^+)$
satisfies
\[
w^* \geq \mu^-\,.
\]
\item[\textit{ii.}] If Problem~\eqref{eqn_problem_pr_wg} is feasible,
then function $w^*= M(\mu^+)$ is its optimal solution.
\end{enumerate}
\end{thrm}

\emph{Proofs of the results.}
Part \textit{i.} follows from Proposition~\ref{prop_feasibility},
part \textit{ii.} follows from Proposition~\ref{prop_solution} (see Section~\ref{sec_proofs}).

\section{Examples}
\label{sec_ex}
As a first example consider the path shown in
Figure~\ref{fig:geometry_path}, whose
curvature is defined as
\begin{equation}
\label{eq:curvature_function}
k(s) = \begin{cases}
0, & \mbox{if } s \in \left[0,l_1\right)\\
k_\tau(s), & \mbox{if } s \in \left[l_1, l_2\right]\\
1/R ,& \mbox{if } s \in  \left(l_2,l_3\right)\\
k_\tau(s), & \mbox{if } s \in \left[l_3, l_4\right]\\
0, & \mbox{if } s \in (l_4,s_f]
\end{cases}
\end{equation}
where $k_\tau(s)$ is the 6-th degree Hermite polynomial used to guarantee
the following interpolation conditions:
\begin{align*}
&k(l_1) = k(l_4) = 0,\\
&k(l_2) = k(l_3) = 1/R,\\
&k^\prime(l_i)=k^{\prime\prime}(l_i)  = 0, \qquad i = 1, \ldots, 4. \\
\end{align*}
\begin{figure}
\centering
\includegraphics[width=2.5in]{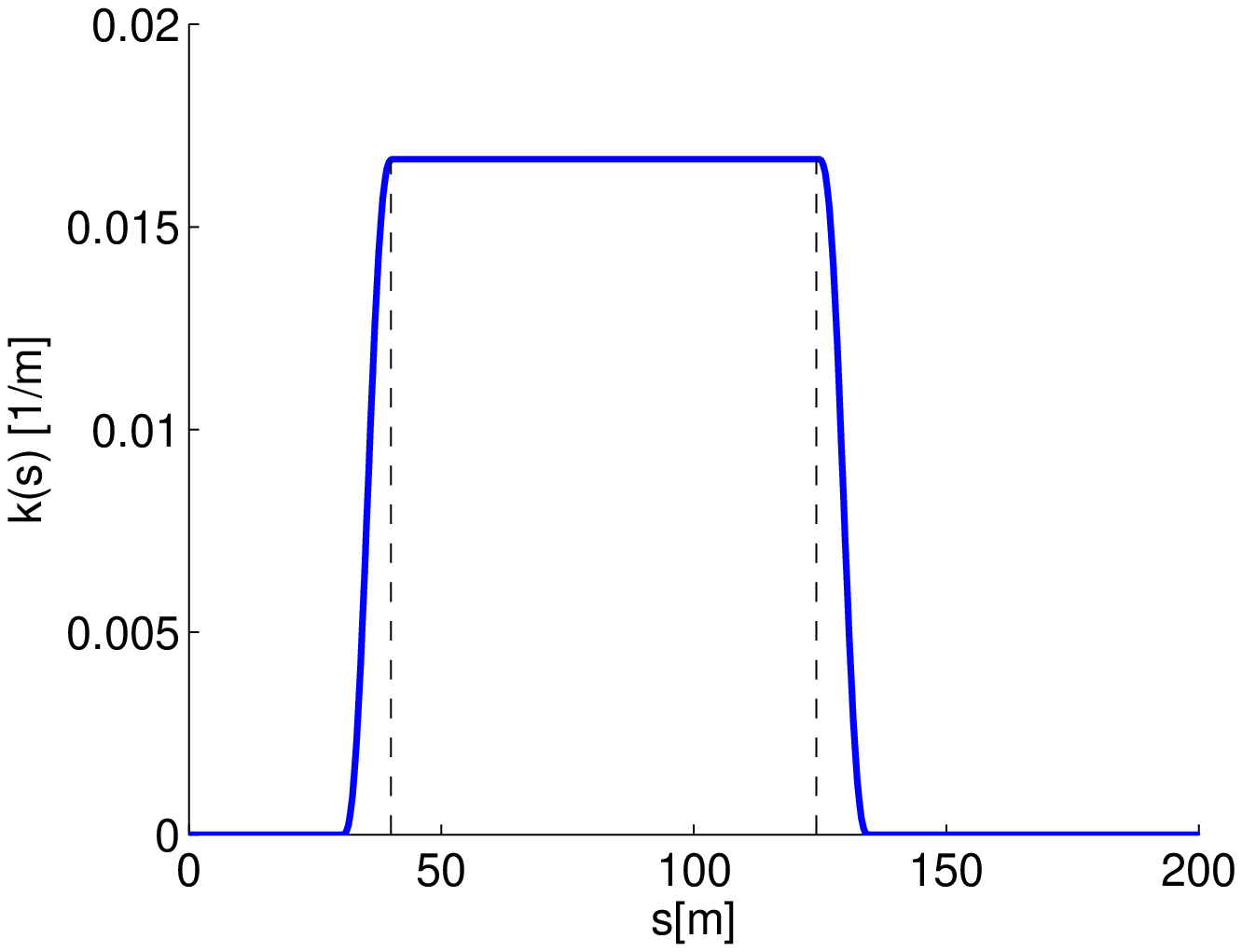}
\includegraphics[width=2.5in]{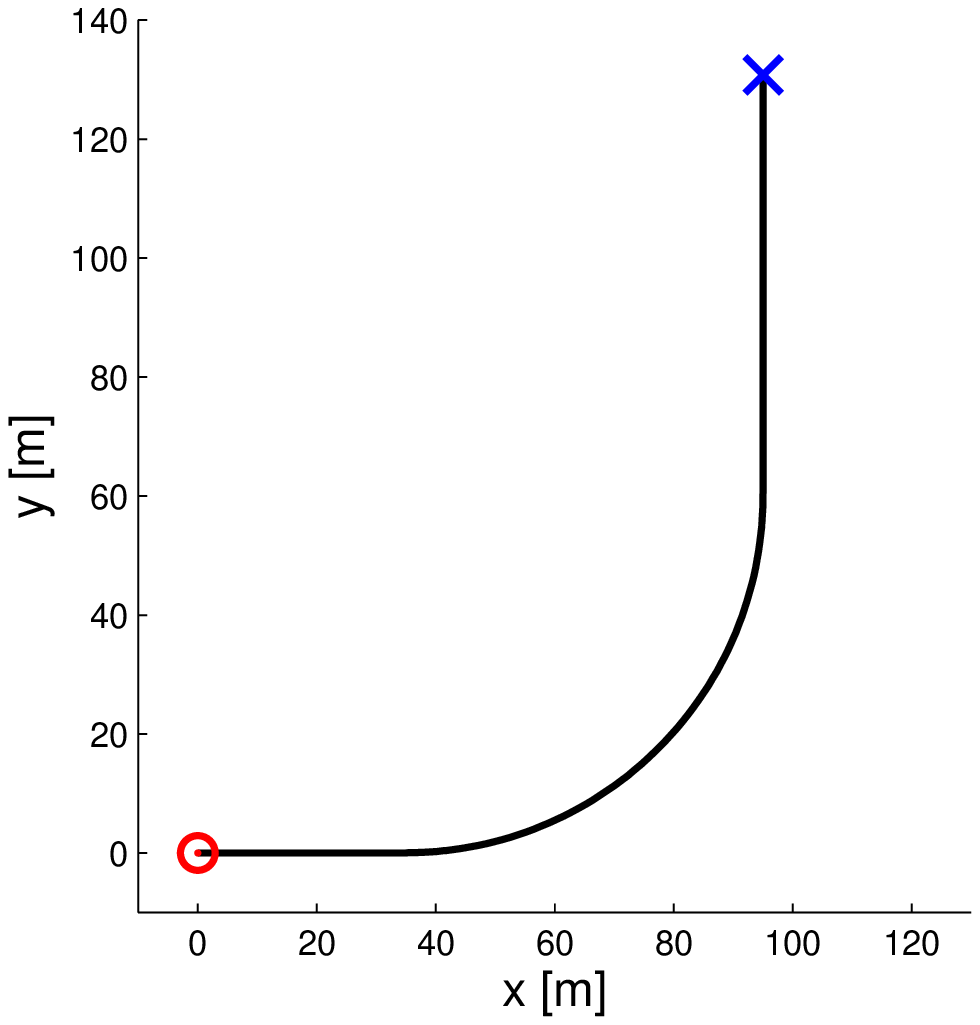}
\caption{On the left, the curvature function $k$
  in~\eqref{eq:curvature_function} of the curve discussed in the first example. On
  the right, the black line represents the path, the red circle and the black cross represent the
  starting point and, respectively, the end point.}
\label{fig:geometry_path}
\end{figure}
In this example, the total length is
$s_f = 200$ and the minimum-time velocity planning problem is addressed with
$\left[l_1,l_2,l_3,l_4\right]=\left[ 30, \ 40,\ 124.2478 ,\ 134.2478\right]$, $R = 60$.
The velocity bounds $v^{+}$ and $v^{-}$ are set as follows: $v^{+}(0) = v^{-}(0) = 0$,
$v^{+}(s_f) = v^{-}(s_f) = 22$, while, for each $s \in \left(0 , s_f \right)$, $v^{-}(s) = 0 $
and $v^{+}(s)=36.1$.
The longitudinal acceleration limits are $\alpha^-= -10.5$ and
$\alpha^+ = 4$, and the maximal normal acceleration is $\beta = 7$.

The following results are obtained by numerically solving
equations~\eqref{eq:forward_iteration},~\eqref{eq:backward_iteration}
with a standart Runge-Kutta 45 integration
scheme. Figure~\ref{fig:ex_operator_f_b} shows the upper-bound
function
$\mu^+$ obtained
by~\eqref{eqn_u} and the corresponding functions $F(u)$ and $B(u)$
computed as the solution of equations~\eqref{eq:forward_iteration} and~\eqref{eq:backward_iteration}, respectively.
Figure~\ref{fig:ex_OPT} shows the optimal solution $w^*$ obtained with \eqref{eq:minimum}.
In this example, the vehicle starts with zero velocity and accelerates
to the upper bound. Then, it follows the velocity bound
in order to respect the
maximum velocity constraint due to the lateral acceleration on the curve.
After that, at the end of the constant bound, the vehicle accelerates and reaches a second local
maximum velocity after which it decelerates quickly in order to reach the final
velocity $v^+(s_f)$.

\begin{figure}
	\centering
	\includegraphics[width=2in]{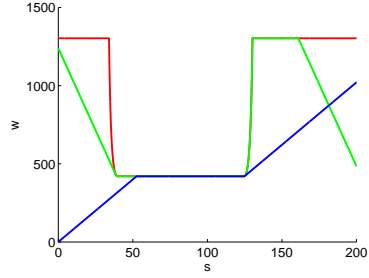}
	\caption{Example 1: The red line represents function $\mu^+$
          defined in~\eqref{eqn_u},  the blue line represents
	$F(u)$ while the green line represents $B(u)$.}
	\label{fig:ex_operator_f_b}
\end{figure}

\begin{figure}[!h]
	\centering
	\includegraphics[width=2in]{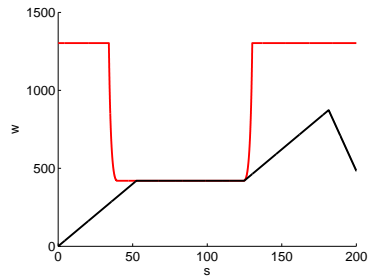}
	\caption{Example 1: The red line
          represents $\mu^+$ defined in~\eqref{eqn_u}, the black line represents
	the optimal solution $w^* = M(u)$.}
	\label{fig:ex_OPT}
\end{figure}

As a second example, consider the same path and constraints as in the
first example, with different initial
and final conditions: $v^{-}(0) = v^{+}(0) = 0$ ,
$v^{-}(s_f) = v^{+}(s_f)=35$.
Figure~\ref{fig:ex_no-feasible} shows function $w^*$ obtained by~\eqref{eq:minimum}.
In this case, Problem~\eqref{eqn_problem_pr_wg} is unfeasible by
Theorem~\ref{main-theorem}, being $w^*(s_f) < v^{-}(s_f)^2$.
In fact, the allowed maximum longitudinal acceleration is not sufficient to
reach the final condition on velocity.\\
\begin{figure}[]
\centering
\includegraphics[ width=2in]{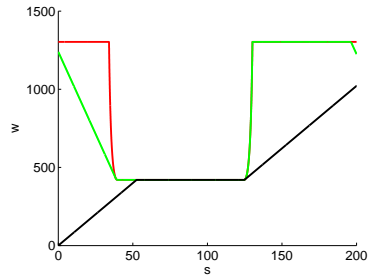}
\caption{Example 2. The green line represents
the velocity function $B(u)(s)$ while the black one depicts function $w^*=M(u)$.
	The final velocity condition is not satisfied: $w^*(s_f) \ne v_f^2$.}
\label{fig:ex_no-feasible}
\end{figure}
As a third example, consider a curve obtained by a quintic
polynomial curve which interpolates coordinates $x = [0,\ 2,\ 2.60,\ 1.75,\ 3]$,
$y = [0 ,\ -0.5,\ 0,\  2, \ 3]$ (see Figure~\ref{fig:path-eta2}).
The velocity planning is addressed with
$v^{+}(0) = v^{-}(0) = 0$, $v^{+}(s_f) = v^{-}(s_f) = 0$ and with
$v^{-}(s) = 0 $ and $v^{+}(s) = 1.3 $ for each $s\in \left(0 , s_f \right)$.
The longitudinal acceleration limits are
$\alpha^- = -0.1$, $\alpha^+ = 0.1$,
and the maximal normal acceleration $\beta= 0.05$.
The resulting optimal velocity profile is plotted in
Figure~\ref{fig:optimal-velocity-example-2}.

\begin{figure}[!h]
\centering
\includegraphics[width = 2.2in]{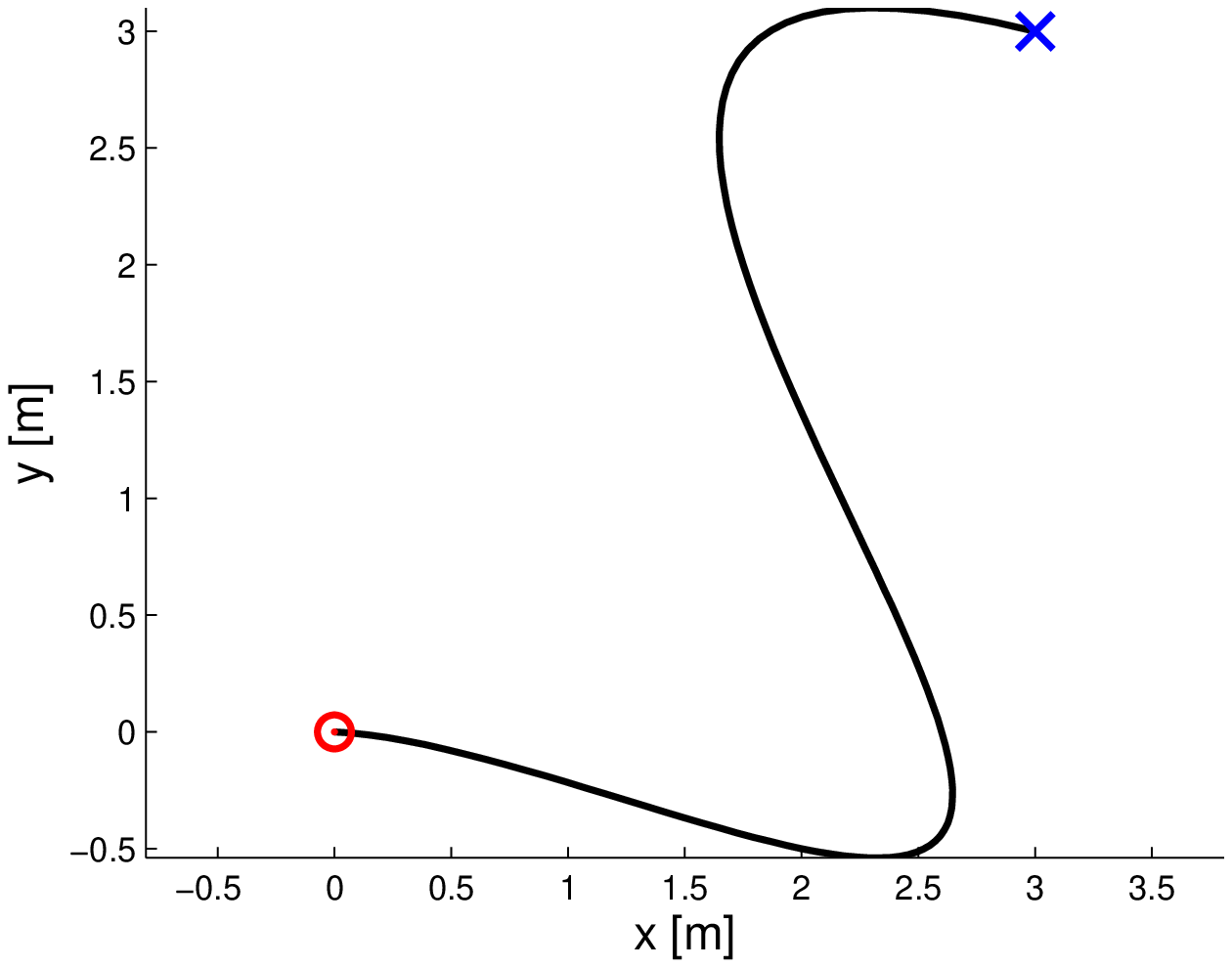}
\caption{A path obtained by quintic-splines interpolation. The black line represents
the path while the circle and the cross represent the start and the end point, respectively.}
\label{fig:path-eta2}
\end{figure}

\begin{figure}[!h]
\centering
\includegraphics[width = 2.2in]{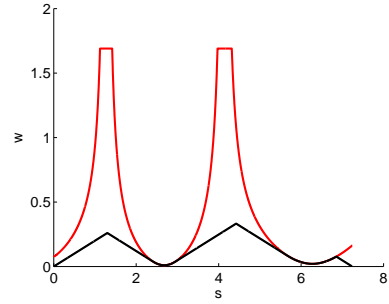}
\caption{The red line represents function $\mu^+$ defined
  in~\eqref{eqn_u}, the black line is the optimal velocity profile $w^*=M(u)$.}
\label{fig:optimal-velocity-example-2}
\end{figure}

\section{Proofs}
\label{sec_proofs}
Given the interval $I = \left[ 0, s_f \right]$,
let $P = \left\{ u \in L^\infty(I) :\ 0 \leq u(s) \leq \left\| \mu^+ \right\|_\infty \text{ for almost every } s \in I \right\}$.
Note that $\langle P;\vee,\wedge \rangle$ is a complete lattice.
Hence, for each subset $S\subseteq P$, there exists a unique least upper bound
$u \in P $, such that, $\forall v \in P$
\[
(\forall w \in S \ w \le v) \iff u \le v.
\]
The least upper bound of $S$ is denoted by $\bigvee S$.
Dually, it is possible to define the greatest lower bound of $S\subseteq P$,
denoted by~$\bigwedge S$ (see Definitions 2.1, 2.4 and Notation 2.3 on
pages 33-34~of~\cite{davey2002introduction}).

Given function $\chi: \Real \rightarrow \{0, 1\}$ defined as follows
\[
\chi(s) =
\begin{cases}
1, & \mbox{ if }\ s \geq 0 \\
0, & \mbox{ otherwise},
\end{cases}
\]
let us define $\forall x, y \in I$, function $A: I \times I \rightarrow \Real$ as
\begin{equation}\label{eq:A}
A(x, y) = \int\limits_{x}^{y} \left\{ \alpha^+(\xi)\chi(y - x) + \alpha^-(\xi)\chi(x - y) \right\} d\xi.
\end{equation}
Define, also, operators $\bar F, \bar B, \bar M:P \to P$, such that,
for $\mu \in P$, $\bar F(\mu)$ and $\bar B(\mu)$ are given as follows
\begin{equation*}\label{def:F}
\begin{cases}
\bar F(\mu)(x) = \underset{y \leq x}{\bigwedge} \{ \mu(y) + A(y, x) \} \\
\bar F(\mu)(0) = \mu(0), \\
\end{cases}
\end{equation*}
\begin{equation*}\label{def:B}
\begin{cases}
\bar B(\mu)(x) = \underset{y \geq x}{\bigwedge} \{ \mu(y) + A(y, x) \} \\
\bar B(\mu)(s_f) = \mu(s_f).
\end{cases}
\end{equation*}
and $\bar M(\mu) = \bar F(\mu) \wedge \bar B(\mu)$.
Observe that $\forall x\in I$,
\begin{equation}
\label{eq:defM}
\bar M(\mu)(x) = \underset{y \leq x}{\bigwedge} \{ \mu(y) + A(y, x) \} \wedge
\underset{y \geq x}{\bigwedge} \{ \mu(y) + A(y, x) \} =
\underset{y \in I}{\bigwedge} \{ \mu(y) + A(y, x) \}.
\end{equation}
As we will show in Proposition~\ref{prop_diff_eq}, the operators we just
introduced are extensions of those defined, respectively,
in~\eqref{eq:forward_iteration},~\eqref{eq:backward_iteration}
and~\eqref{eq:minimum}.

We will refer to the following definitions.

\begin{dfntn}
An operator $\phi: P \to P$ is meet preserving if, $\forall u, v \in P$,
\[
\phi(u \wedge v)=\phi(u) \wedge \phi(v).
\]
\end{dfntn}

\begin{dfntn}
An operator $\phi: P \to P$ is order preserving if, $\forall u, v \in P$,
\[
u \leq v \Rightarrow \phi(u) \leq \phi(v).
\]
\end{dfntn}

\noindent
Even though we will only use the fact that operators $\bar F, \bar B$ and $\bar M$
are order preserving, for completeness we state the following Proposition.

\begin{prpstn}\label{prop_meetpreserving}
Operators $\bar F$, $\bar B$ and $\bar M$ are meet preserving and order preserving.
\end{prpstn}

\begin{proof}
Let $u, v \in P$, $w= u \wedge v$.
We want to show that $\bar F(w) = \bar F(u) \wedge \bar F(v)$.
By definition of $\bar F$ we have that $\bar F(w)(0) = \bar F(u)(0) \wedge \bar F(v)(0)$ and
$\forall x \in I$,
\begin{align*}
\bar F(w)(x) & = \bigwedge_{y \leq x} \left\{ w(y) + A(y, x) \right\} =
\bigwedge_{y \leq x} \left\{ (u \wedge v)(y) + A(y, x) \right\} =
\bigwedge_{y \leq x} \left\{ u(y) \wedge v(y) + A(y, x) \right\} = \\
& = \bigwedge_{y \leq x} \left\{ u(y) + A(y, x) \right\} \wedge
\bigwedge_{y \leq x} \left\{ v(y) + A(y, x) \right\} = \bar F(u)(x) \wedge \bar F(v)(x).
\end{align*}
The proof that $\bar B(u \wedge v) = \bar B(u) \wedge \bar B(v)$ is analogous.
Finally, $\bar M(u \wedge v) = \bar F(u \wedge v) \wedge \bar B(u \wedge v)
= \bar F(u) \wedge \bar F(v) \wedge \bar B(u) \wedge \bar B(v) =
\bar F(u) \wedge \bar B(u) \wedge \bar F(v) \wedge \bar B(v) = \bar M(u) \wedge \bar M(v)$.
Since maps $\bar F, \bar B, \bar M$ are meet preserving, they are also order preserving
(see Proposition~2.19 on page 44 of~\cite{davey2002introduction}).
\end{proof}

\begin{prpstn}\label{prop:A}
Function $A: I \times I \rightarrow \Real$ defined as in~\eqref{eq:A} is a
hemi-metric, that is, it satisfies the following properties:
\begin{itemize}
\item[\textit{i.}] $\forall x, y \in I,\ A(x, y) \geq 0$,
\item[\textit{ii.}] $\forall x \in I,\ A(x, x) = 0$,
\item[\textit{iii.}] $\forall x, y, z \in I,\ A(x, z) \leq A(x, y) + A(y, z)$ (i.e., the triangular inequality holds).\\
Moreover, equality holds if $x \geq y \geq z$ or $x \leq y \leq z$.
\end{itemize}
\end{prpstn}

\begin{proof}
\begin{itemize}
\item[\textit{i.}]
It holds, since $\alpha^+$ is non-negative
and $\alpha^-$ is non-positive over $I$.
\item[\textit{ii.}]
It holds trivially by definition of $A$.
\item[\textit{iii.}]
For $y \geq z \geq x$:
\begin{align*}
A(x, z) = & \int\limits_{x}^{z} \left\{ \alpha^+(\xi)\chi(z - x) + \alpha^-(\xi)\chi(x - z) \right\} d\xi =
\int\limits_{x}^{z} \alpha^+(\xi) d\xi \leq \\
\leq & \int\limits_{x}^{y} \alpha^+(\xi) d\xi \leq
\int\limits_{x}^{y} \alpha^+(\xi) d\xi - \int\limits_{z}^{y} \alpha^-(\eta) d\eta =
\int\limits_{x}^{y} \alpha^+(\xi) d\xi + \int\limits_{y}^{z} \alpha^-(\eta) d\eta = \\
= & \int\limits_{x}^{y} \left\{ \alpha^+(\xi)\chi(y - x) + \alpha^-(\xi)\chi(x - y) \right\} d\xi +
\int\limits_{y}^{z} \left\{ \alpha^+(\eta)\chi(z - y) + \alpha^-(\eta)\chi(y - z)\right\} d\eta = \\
= & A(x,y)+A(y,z).
\end{align*}
The same reasoning applies also to the case when $z \geq x \geq y$, $x \geq z \geq y$
or $y \geq x \geq z$.
Next, let us show that equality holds for any $x \leq y \leq z$:
\begin{align*}
A(x, z) = & \int\limits_{x}^{z} \left\{ \alpha^+(\xi)\chi(z - x) + \alpha^-(\xi)\chi(x - z) \right\} d\xi =
\int\limits_{x}^{z} \alpha^+(\xi) d\xi =
\int\limits_{x}^{y} \alpha^+(\xi) d\xi + \int\limits_{y}^{z} \alpha^+(\eta) d\eta = \\
= & \int\limits_{x}^{y} \left\{ \alpha^+(\xi)\chi(y - x) + \alpha^-(\xi)\chi(x - y) \right\} d\xi +
\int\limits_{y}^{z} \left\{ \alpha^+(\eta)\chi(z - y) + \alpha^-(\eta)\chi(y - z) \right\} d\eta = \\
= & A(x, y) + A(y, z).
\end{align*}
The proof that the equality holds also for any $x \geq y \geq z$ is analogous.
\end{itemize}
\end{proof}

\begin{prpstn}
\label{prop_prop_of_M}
Function $\bar M$ satisfies the following properties, $\forall \mu \in P$,
\begin{itemize}
\item[\textit{i.}] $\bar M(\mu) \leq \mu$,
\item[\textit{ii.}] $\bar M^2(\mu) = \bar M(\mu)$, where $\bar M^2(\mu)$
stands for $\bar M(\bar M(\mu))$.
\end{itemize}
\end{prpstn}

\begin{proof}
\begin{itemize}
\item[\textit{i.}] It is a consequence of the definition of $\bar M$.
\item[\textit{ii.}] Let us now show that $\bar F(\bar M(\mu)) = \bar M(\mu)$:
the fact that $\bar F(\bar M(\mu)) \leq \bar M(\mu)$ follows by the definition of $\bar F$ whilst,
to prove the opposite inequality, note that, by Proposition~\ref{prop:A} and
\eqref{eq:defM},
\begin{align*}
\bar F(\bar M(\mu))(x) = & \bigwedge_{y \leq x} \left\{ \bar M(\mu)(y) + A(y, x) \right\} =
\bigwedge_{y \leq x} \left\{ \bigwedge_{z \in I} \left\{ \mu(z) + A(z, y) \right\} + A(y, x) \right\} \geq \\
\geq & \bigwedge_{z \in I} \left\{ \mu(z) + A(z, x) \right\} = \bar M(\mu)(x).
\end{align*}
In the same way it can be proved that $\bar B(\bar M(\mu)) = \bar M(\mu)$, from which it follows that
$\bar M(\bar M(\mu)) = \bar M(\mu)$.
\end{itemize}
\end{proof}

\begin{prpstn}
\label{prop_max_property_M}
\[
\bar M(\mu^+) = \bigvee \left\{ u \in P : u \leq \bar M(u), u \leq \mu^+ \right\}
\]
\end{prpstn}

\begin{proof}
Set $U = \left\{ u \in P : u \leq \mu^+ \right\}$.
Note that $\langle U; \vee, \wedge \rangle$ is a sublattice of $\langle P; \vee, \wedge \rangle$,
moreover, by~\emph{i.}~of Proposition~\ref{prop_prop_of_M}, if $u \in U$,
then $\bar M(u) \in U$.
Since $\bar M$ is order preserving by Proposition~\ref{prop_meetpreserving},
by the Knaster-Tarski Fixpoint Theorem
(Theorem~2.35 on page 50 in~\cite{davey2002introduction})
\begin{equation*}
u^* = \bigvee \left\{ u \in P: u \leq \bar M(u), u \leq \mu^+ \right\}
\end{equation*}
is such that $u^*$ is the greatest fixed point of $\bar M$ such that $u^* \in U$.
Let $u = \bar M(\mu^+)$, by part~\emph{ii.}~of Proposition~\ref{prop_prop_of_M}, we know
that $u$ is also a fixed point of $\bar M$, thus, by definition of $u^*$, $u^* \geq u$.
To prove that $u^* = u$, that is, to prove that $u$ is also the greatest fixed point,
it remains to show that $u^* \leq u$.
To this end, assume, by contradiction, that $u^* \nleq u$.
Since $u^* = \bar M(u^*)$, $u = \bar M(\mu^+)$ and the fact that $\bar M$ is order preserving,
it follows that $u^*\nleq \mu^+$, which contradicts the definition of $u^*$ $\lightning$.
\end{proof}

\begin{rmrk}
Given $u, v \in P$, if $u \nleq v$, this does not imply that $u \geq v$ and $u \neq v$,
as $u$ and $v$ may not be comparable with respect to partial order $\leq$.
\end{rmrk}

\begin{prpstn}
\label{prop_intersec}
The following two statements are equivalent:
\begin{itemize}
\item[\textit{i.}] Set $\left\{ u \in P : u = \bar M(u), \ \mu^- \leq u \leq \mu^+ \right\}$ is not empty.
\item[\textit{ii.}] $\bar M(\mu^+) \geq \mu^-$.
\end{itemize}
\end{prpstn}

\begin{proof}
\begin{itemize}
\item[]
\item[$\left.\textit{ii.}\Rightarrow\textit{i.}\right)$]
It follows from the fact that $u^* = \bar M(\mu^+)$ is such that
$\bar M(u^*)=u^*$ by part~\emph{ii.} of Proposition~\ref{prop_prop_of_M}.

\item[$\left.\textit{i.}\Rightarrow\textit{ii.}\right)$]
By contradiction, assume that $\bar M(\mu^+) \ngeq \mu^-$.
Choose any $w \in P$ such that $w = \bar M(w)$ and $w \leq \mu^+$.
By Proposition~\ref{prop_max_property_M},
$\bar M(w) \leq \bar M(\mu^+)$, hence $w \leq \bar M(\mu^+)$,
but $\bar M(\mu^+) \ngeq \mu^-$, so $w \ngeq \mu^-$.
Thus, being $w$ any fixed point of $\bar M$ such that $w \leq \mu^+$,
set $\left\{u \in P: u = \bar M(u), \ \mu^- \leq u \leq \mu^+ \right\}$ is empty $\lightning$.
\end{itemize}
\end{proof}

\begin{prpstn}
\label{prop_diff_eq}
If $\mu \in \mathcal{Q}$,
then $\bar F(\mu), \bar B(\mu) \in W^{1,\infty}(I)$ satisfy a. e.
\begin{equation}\label{prop:F_diff_eq}
\begin{cases}
\bar F(\mu)^\prime(x) =
\begin{cases}
\alpha^+(x) \wedge \mu^\prime (x), & \mbox{if } \ \bar F(\mu)(x) \geq \mu(x) \\
\alpha^+(x), & \mbox{if } \ \bar F(\mu)(x) < \mu(x)
\end{cases}\\
\bar F(\mu)(0) = \mu(0),
\end{cases}
\end{equation}
and
\begin{equation}\label{prop:B_diff_eq}
\begin{cases}
\bar B(\mu)^\prime(x) =
\begin{cases}
\alpha^-(x) \vee \mu^\prime (x), & \mbox{if } \ \bar B(\mu)(x) \geq \mu(x) \\
\alpha^-(x), & \mbox{if } \ \bar B(\mu)(x) < \mu(x)
\end{cases}\\
\bar B(\mu)(s_f) = \mu(s_f).
\end{cases}
\end{equation}
\end{prpstn}

\begin{proof}
Let $J = \{ x \in I : \sgn(\mu^\prime - \alpha^+) \mbox{ is continuous at } x\}$.
Note that, since $\mu \in \mathcal{Q}$, $J$ contains almost all elements of $I$.
Let $x \in J$, then
\begin{gather}
\nonumber
\lim_{h \rightarrow 0^+} \frac{\bar F(\mu)(x + h) - \bar F(\mu)(x)}{h}
= \lim_{h \rightarrow 0^+} h^{-1}
\!\!\left[ \underset{y \leq x + h}{\bigwedge} \{ \mu(y) + A(y, x + h) \} - \bar F(\mu)(x) \right] = \\
\label{eq:limit}
= \lim_{h \rightarrow 0^+} h^{-1} \left[ \left( \underset{y \leq x}{\bigwedge}
\{ \mu(y) + A(y, x + h) \} - \bar F(\mu)(x) \right) \right. \wedge
\left. \left( \underset{x < y \leq x + h}{\bigwedge}
\{ \mu(y) + A(y, x + h) \} - \bar F(\mu)(x) \right) \right]
\end{gather}
Since $A(y, x+h) = A(y, x) + A(x, x+h)$ by Proposition~\ref{prop:A},
the first parenthesis of~\eqref{eq:limit} reduces to
\begin{equation*}\label{propr:par1}
\underset{y \leq x}{\bigwedge} \{ \mu(y) + A(y, x + h) \} - \bar F(\mu)(x) =
\bar F(\mu)(x) + A(x, x + h) - \bar F(\mu)(x) = A(x, x + h).
\end{equation*}
Being $\sgn(\mu^\prime - \alpha^+)$ continuous at $x$, it is possible to choose $h > 0$
sufficiently small such that $\sgn(\mu^\prime - \alpha^+)$ is constant on interval
$[x, x+h]$.
Set $J^+ = \left\{x \in J : \mu^\prime(x) - \alpha^+(x) > 0 \right\}$ and
$J^- = J \setminus J^+$.
Then, the second parenthesis of~\eqref{eq:limit} can be rewritten as
\begin{equation*}
\underset{x < y \leq x + h}{\bigwedge} \left\{ \mu(y) + A(y, x + h) \right\} - \bar F(\mu)(x) =
\begin{cases}
\mu(x + h) - \bar F(\mu)(x), & \mbox{if } \ x \in J^- \\
\mu(x) + A(x, x+h) - \bar F(\mu)(x), & \mbox{if } \ x \in J^+,
\end{cases}
\end{equation*}
since in the former case the minimum of $\mu(y) + A(y, x + h)$ over
$[x, x+h]$ is attained at $x+h$, whilst in the latter is attained at $x$.

Hence, we have that
\begin{align*}
\lim_{h \rightarrow 0^+} \frac{\bar F(\mu)(x + h) - \bar F(\mu)(x)}{h}
= & \begin{cases}
\lim\limits_{h \rightarrow 0^+} h^{-1}\left[ A(x, x + h) \wedge
\left( \mu(x + h) - \bar F(\mu)(x) \right) \right], & \mbox{if } \ x \in J^- \\
\lim\limits_{h \rightarrow 0^+} h^{-1}\left[ A(x, x + h) \wedge
\left( \mu(x) + A(x, x+h) - \bar F(\mu)(x) \right) \right], & \mbox{if } \ x \in J^+
\end{cases}\\
= & \begin{cases}
\alpha^+(x) \wedge \lim\limits_{h \rightarrow 0^+} h^{-1} \left( \mu(x + h) - \bar F(\mu)(x) \right),
& \mbox{if } \ x \in J^- \\
\alpha^+(x) \wedge \lim\limits_{h \rightarrow 0^+} h^{-1} \left( \mu(x) + A(x, x+h) - \bar F(\mu)(x) \right),
& \mbox{if } \ x \in J^+
\end{cases}\\
= & \begin{cases}
\alpha^+(x) \wedge \mu^\prime(x),
& \mbox{if } \ x \in J^- \ \mbox{ and } \ \bar F(\mu)(x) \geq \mu(x) \\
\alpha^+(x) \wedge +\infty = \alpha^+(x),
& \mbox{if } \ x \in J^- \ \mbox{ and } \ \bar F(\mu)(x) < \mu(x) \\
\alpha^+(x) = \alpha^+(x) \wedge \mu^\prime(x),
& \mbox{if } \ x \in J^+ \ \mbox{ and } \ \bar F(\mu)(x) \geq \mu(x) \\
\alpha^+(x) \wedge +\infty = \alpha^+(x),
& \mbox{if } \ x \in J^+ \ \mbox{ and } \ \bar F(\mu)(x) < \mu(x)
\end{cases}\\
= &
\begin{cases}
\alpha^+(x) \wedge \mu^\prime(x), & \mbox{if } \bar F(\mu)(x) \geq \mu(x) \\
\alpha^+(x), & \mbox{if } \bar F(\mu)(x) < \mu(x).
\end{cases}
\end{align*}
Note that, by definition of $\bar F$, $\bar F(\mu)(x) \leq \mu(x)$ must hold.
In conclusion, we proved that $\bar F(u) \in W^{1,\infty}(I)$
and satisfies~\eqref{prop:F_diff_eq}. \\
Applying the same reasoning it can be proved that $\bar B(u) \in W^{1,\infty}(I)$
and satisfies~\eqref{prop:B_diff_eq}.
\end{proof}

\begin{prpstn}
\label{prop_feas}
Assume that $\mu^+ \in \mathcal{Q}$ and let $w \in P$,
then $w$ is feasible for Problem~\eqref{eqn_problem_pr_wg}
(i.e., it satisfies constraints \eqref{con_speed_pr_wg} and \eqref{con_at_pr_wg}),
if and only if $\mu^- \leq w \leq \mu^+$ and $\bar M(w) = w$.
\end{prpstn}

\begin{proof}
\begin{itemize}
\item[$\left.\Rightarrow\right)$]
Assume that $w$ is feasible for
 Problem~\eqref{eqn_problem_pr_wg}, then $w$ satisfies a. e.
$w^\prime \leq \alpha^+$. Thus, $\phi = w$ is the solution
of~\eqref{eq:forward_iteration} for $\mu = w$, which implies that
$\bar F(w)=w$.
Analogously $\bar B(w)=w$, so that $\bar M(w)=w$.
Moreover, since $w$ satisfies the bounds of Problem~\eqref{eqn_problem_pr_wg},
it follows that $\mu^- \leq w \leq \mu^+$.

\item[$\left.\Leftarrow\right)$]
Condition~\eqref{con_speed_pr_wg} holds by hypothesis.
Since $\bar M(w) = w$, then it must be $\bar F(w) = w$. In fact, if
by contradiction $\bar F(w) < w$, $\bar M(w) \leq \bar F(w) < w$, which contradicts
the assumption $\lightning$.
Then $\bar F(w) = w$, implies that, a. e., $\bar F(w)' = w'$ which
by definition of $\bar F(w)'$ in~\eqref{eq:forward_iteration}, implies that, a. e., $w'
\leq \alpha^+$.
Analogously, it must be $\bar B(w) = w$ which implies that, a. e. $w' \geq
\alpha^-$ and condition~\eqref{con_at_pr_wg} holds.
\end{itemize}
\end{proof}

\begin{prpstn}
\label{prop_feasibility}
Assume that $\mu^+ \in \mathcal{Q}$.
Then, Problem~\eqref{eqn_problem_pr_wg} is feasible if and only if $\bar M(\mu^+)\geq \mu^-$.
\end{prpstn}

\begin{proof}
\begin{itemize}
\item[$\left.\Rightarrow\right)$]
Let $w$ be a feasible solution of
Problem~\eqref{eqn_problem_pr_wg}. Then, by
Proposition~\ref{prop_feas}, $\bar M(w)=w \leq \mu^+$.
Hence, being $\bar M$ order preserving,  $\bar M(\mu^+) \geq \bar M(w) \geq \mu^-$.

\item[$\left.\Leftarrow\right)$]
$w=\bar M(\mu^+)$ satisfies $\bar M(w) = w$ (by part \textit{ii.} of
Proposition~\ref{prop_prop_of_M}) and $\mu^+ \leq w \leq \mu^-$ (by
hypothesis). Hence, by Proposition~\ref{prop_feas}, $w$ is a feasible
solution of  Problem~\eqref{eqn_problem_pr_wg}.
\end{itemize}
\end{proof}

\begin{prpstn}
\label{prop_solution}
If $\mu^+ \in \mathcal{Q}$ and Problem~\eqref{eqn_problem_pr_wg}
is feasible, then $w^*=\bar M(\mu^+)$ is its optimal solution.
\end{prpstn}

\begin{proof}
By contradiction, assume that there exists a feasible $\tilde w$ such that
$\Psi(\tilde w) < \Psi(w^*)$. Since $\tilde w$ is feasible,
Proposition~\ref{prop_feas} implies that $\bar M(\tilde w)=\tilde w$.
Moreover, since $\Psi$ is order reversing, $\tilde w \nleq w^*$.
This is not possible since $w^*=\bigvee \{w \in P: w \leq \bar M(w), w \leq
\mu^+\}  \geq \tilde w$, by Proposition~\ref{prop_max_property_M} $\lightning$.
\end{proof}

\bibliographystyle{abbrv}
\bibliography{Automatica_PR,VelPlan}

\end{document}